\documentclass[11pt]{amsart}
\usepackage{paralist}
\usepackage{amssymb}
\usepackage{amstext}
\usepackage{amsmath}
\usepackage{amscd}
\usepackage{latexsym}
\usepackage{amsfonts}
\usepackage{color}
\usepackage{cancel}
\theoremstyle{plain}
\newtheorem{thm}{Theorem}[section]
\newtheorem*{thm*}{Theorem}
\newtheorem*{cor*}{Corollary}

\newtheorem{prop}[thm]{Proposition}
\newtheorem{lem}[thm]{Lemma}

\newtheorem{cor}[thm]{Corollary}
\newtheorem{claim}{Claim}
\newtheorem*{claim*}{Claim}

\theoremstyle{definition}
\newtheorem{defn}[thm]{Definition}
\newtheorem{ex}[thm]{Example}
\newtheorem{rem}[thm]{Remark}

\theoremstyle{remark}

\numberwithin{equation}{thm}
\def\Im{\mathrm{Im}}

\def\rank{\mathrm{rank}}
\def\a{\mathfrak a}

\def\e{\mathrm{e}}
\def\m{\mathfrak m}

\def\p{\mathfrak p}

\def\Z{\Bbb Z}

\def\H{\mathrm{H}}

\newcommand{\rma}{\mathrm{a}}

\newcommand{\rmG}{\mathrm{G}}

\newcommand{\ov}{\overline}

\def\type{\mathrm{type}}
\def\depth{\mathrm{depth}}

\def\Ann{\mathrm{Ann}}
\def\Ass{\mathrm{Ass}}

\def\height{\mathrm{ht}}

\def\R{{\mathcal R}}

\newcommand{\mapright}[1]{%
\smash{\mathop{%
\hbox to 1cm{\rightarrowfill}}\limits^{#1}}}

\newcommand{\mapleft}[1]{%
\smash{\mathop{%
\hbox to 1cm{\leftarrowfill}}\limits_{#1}}}

\newcommand{\mapdown}[1]{\Big\downarrow
\llap{$\vcenter{\hbox{$\scriptstyle#1\,$}}$ }}

\begin{document}

\setlength{\baselineskip}{12.2pt}
%%%%%%%%%%%%%%%%%%%%%%%%%%%%%%%%%%%%%%%%%%%%%%%%%%%%%%%%%%%%%
%%%%%%%%%%%%%%%%%%%%%%%%%%%%%%%%%%%%%%%%%%%%%%%%%%
\title[Sally module and the second normal Hilbert coefficient]{On the structure of the Sally module and the second normal Hilbert coefficient}
 
\pagestyle{plain}
\author[Masuti]{Shreedevi K. Masuti}
\address{Chennai Mathematical Institute, Siruseri, Tamilnadu 603103. India}
\email{shreedevikm@cmi.ac.in}
\author[Ozeki]{Kazuho Ozeki}
\address{Department of Mathematical Sciences, Faculty of Science, Yamaguchi University, 1677-1 Yoshida, Yamaguchi 753-8512, Japan}
\email{ozeki@yamaguchi-u.ac.jp}
\author[Rossi]{Maria Evelina Rossi}
\address{Dipartimento di Matematica, Universit{\`a} di Genova, Via Dodecaneso 35, 16146 Genova, Italy}
\email{rossim@dima.unige.it}
\author[Truong]{Hoang Le Truong}
\address{Institute of Mathematics, Vietnam Academy of Science and Technology, 18 Hoang Quoc Viet Road, 10307 Hanoi, Vietnam }
\address{Mathematik und Informatik, Universit\"{a}t des Saarlandes, Campus E2 4,D-66123 Saarbr\"{u}cken, Germany}
\email{hltruong@math.ac.vn\\
	hoang@math.uni-sb.de}

%\date{\today}
\thanks{SKM is supported by INSPIRE faculty award funded by Department of Science and Technology, Govt. of India. She is also partially supported by a grant from Infosys Foundation. She was supported by INdAM COFUND Fellowships cofunded by Marie Curie actions, Italy, for her research in Genova during which the work has started. KO was partially supported by Grant-in-Aid for Scientific Researches (C) in Japan (18K03241). MER was partially supported by PRIN 2015EYPTSB-008 Geometry of Algebraic Varieties.
The fourth author was partially supported by  the Vietnam Institute for Advanced Study in Mathematics and the Vietnam National Foundation for Science and Technology Development (NAFOSTED) under grant number 101.04-2017.14.
}
\maketitle
%%%%%%%%%%%%%%%%%%%%%%%%%%%%%%%%%%%%%%%%%%%%%%%%%%%%%%%%%%%%%
%%%%%%%%%%%%%%%%%%%%%%%%%%%%%%%%%%%%%%%%%%%%%%%%%%%%%%%%%%%%%

\begin{abstract}
The Hilbert coefficients of the normal filtration  give important geometric information on the base ring like the pseudo-rationality. The Sally module was introduced by W.V. Vasconcelos and it is useful to connect the Hilbert coefficients to the homological properties of the  associated graded module of a Noetherian filtration. In this paper we give a complete structure of the Sally module in the case the second normal Hilbert coefficient attains almost minimal value in an analytically unramified Cohen-Macaulay local ring. As a consequence, in this case we present a complete description of the Hilbert function of the associated  graded ring of the normal filtration.  A deep analysis of the vanishing of the third Hilbert coefficient has been necessary. This study is  related to a long-standing conjecture stated by S.~Itoh.  
% bound on $\ov{\e}_2(I) - \ell_R(\ov{I^2}/ J \ov{I})$ given in \cite{KM15} in the case $\ov{\e}_3(I)=0.$ 

% In this paper we give a characterization of the almost minimal value of the second normal Hilbert coefficient in an analytically unramified Cohen-Macaulay local ring.
%The Sally module of an ideal is an important tool to interplay between Hilbert coefficients and the properties of the associated graded ring. In this paper we give new insights on the structure of the Sally module which characterizes the almost minimal value of the first Hilbert coefficient in the case of the integrally closure filtration in an analytically unramified Cohen-Macaulay local ring $R. $ %The first normal Hilbert coefficient of a primary ideal plays an important role in commutative algebra and in algebraic geometry. In this paper we give a complete algebraic structure of the normal associated graded ring of $\m$-primary ideals $I$ in an analytically unramified Cohen-Macaulay local ring $R$ satisfying the equality $\ov{\e}_1(I)=\ov{\e}_0(I)-\ell_R(R/\ov{I})+\ell_R(\ov{I^2}/J\ov{I})+1, $ where $J$ is a minimal reduction of $I$, and $\ov{\e}_0(I)$ and $\ov{\e}_1(I)$ denote the multiplicity and the first normal Hilbert coefficient of $I,$ respectively. 
\end{abstract}

{\footnotesize
% \tableofcontents
}

%%%%%%%%%%%%%%%%%%%%%%%%%%%%%%%%%%%%%%%%%%%%%%%%%%%%%%%%%%%%%
%%%%%%%%%%%%%%%%%%%%%%%%%%%%%%%%%%%%%%%%%%%%%%%%%%%%%%%%%%%%%
%%%%%%%%%%%%%%%%%%%%%%%%%%%%%%%%%%%%%%%%%%%%%%%%%%%%%%%%%%%%%
%%%%%%%%%%%%%%%%%%%%%%%%%%%%%%%%%%%%%%%%%%%%%%%%%%%%%%%%%%%%%
%%%%%%%%%%%%%%%%%%%%%%%%%%%%%%%%%%%%%%%%%%%%%%%%%%%%%%%%%%%%%
%%%%%%%%%%%%%%%%%%%%%%%%%%%%%%%%%%%%%%%%%%%%%%%%%%%%%%%%%%%%%

\section{Introduction}
The study of  the homological properties of the blow-up algebras of an ideal in a Noetherian local ring is an important problem in commutative algebra and in algebraic geometry. The problem is difficult and it  is one of the main obstacles in the resolution of singularities. As a remedy one tries to find available information   from  the Hilbert polynomial with respect to suitable filtrations.  The filtration  of the integral closure of the powers of an ideal gives rise to the normal Hilbert polynomial, first investigated by D. Rees in his study of pseudo-rational local rings. This study was carried on by J. Lipman, B. Teissier and more recently by   T. Okuma, K.-i. Watanabe, and K.Yoshida with the study of $p_g$-ideals  which inherit  nice properties of integrally closed ideals on rational singularities, see \cite{OWY17, OWY19}.
%Often if these invariants achieve extremal values with respect to some bounds, then the blow-up algebra of the corresponding ideal has good homological properties. This is a reason for which there is a huge interest in the literature in finding   sharp bounds on these numerical invariants, see for instance \cite{RV10} and references therein.   
%\vskip 2mm    The main goal of this paper is to explore the second Hilbert coefficient associated to this filtration.  
\vskip 2mm
Let $(R,\m)$ be an analytically unramified Cohen-Macaulay local ring of dimension $d>0$ with infinite residue field $R/\m$ and $I$ an $\m$-primary ideal of $R.$ Let $\overline{I}$ denote the integral closure of $I.$ By \cite{Ree61}  it is well known that if we consider  the normal filtration $\{\ov{I^n}\}_{n \in \mathbb{Z}},$  there exist integers $\ov{\e}_i(I)$, called the {\it normal Hilbert coefficients} of $I$ such that
\begin{eqnarray*}
 \ell_R(R/\ov{I^{n+1}})=\ov{\e}_0(I) \binom{n+d}{d}-\ov{\e}_1(I) \binom{n+d-1}{d-1}+\cdots+(-1)^d \ov{\e}_d(I)
\end{eqnarray*}
for $n \gg 0$.
Here $\ell_R(N)$ denotes, for an $R$-module $N$, the length of $N$.  
A large literature was devoted to the study of the integer $\ov{g}_s(I):=\ov{\e}_1(I)-\ov{\e}_0(I)+\ell_R(R/\ov{I}), $  called by Ooishi   the normal sectional genus of $I$ (see \cite{O87}). 

It is well known that $\ov{g}_s(I) \ge 0$ and if the equality holds, then  the normal associated graded ring $\ov{G}(I):=\oplus_{n \geq 0} \ov{I^n}/\ov{I^{n+1}}$ of $I$ is Cohen-Macaulay and $\ov{I^{n+1}}=J\ov{I^n}$ for all  $n\ge 1 $  where $J$ is any minimal reduction of $I $ $($\cite{  H87, Itoh92, N60, O87}$)$. Notice that to prove the last equality it is not enough to prove $\ov{I^{2}}=J\ov{I}$ as for the $I$-adic filtration.   Okuma produced an interesting geometrical  example  with   $\ov{I^{2}}=J\ov{I },$ but  $\ov{I^{3}}\neq J\ov{I^2}$  (private communication).  This makes it clear how difficult can be  in general  to get information on the  reduction number $\ov{\rm{r}}_J(I):=\min \{r \geq 0 \ | \ \mbox{$\ov{I^{n+1}}=J\ov{I^n}$ for all $n \geq r$} \}, $ an important numerical invariant of $I.$ Recently   A. Corso, C. Polini and M. E. Rossi showed that if   $\ov{g}_s(I)=1$ holds true, then $\depth ~\ov{G}(I) \geq d-1$   (see {CPR16}).   
This case was also explored by T. T. Phuong \cite{Phu18} when $R$ is a generalized Cohen-Macaulay ring. 

  The normal sectional genus is strictly related to the second normal Hilbert coefficient. 
%When $d=2,$ then $\ov{\e}_2(I)$ is called the normal genus of $I $  and Narita proved that $\ov{\e}_2(I)$ is always non-negative.    \vskip 2mm
By \cite[Theorem 2]{Itoh92}  the following inequalities
$$ \ov{\e}_2(I) \geq \ov{\e}_1(I) -\ov{\e}_0(I)+\ell_R(R/\ov{I})  {=\ov{g}_s(I)} \geq \ell_R(\ov{I^2}/J\ov{I})  $$
hold true and if either of the inequalities is an equality, then $\ov{\rm{r}}_J(I) \le 2,  $  in particular $\ov{G}(I) $ is Cohen-Macaulay (see also \cite[Theorem 3.11]{CPR05}). The vanishing of $ \ov{\e}_2(I)  $ is a particular case of this situation. In an excellent normal domain of dimension two, $ \ov{\e}_2(I)=0  $ characterizes the $p_g$-ideals, see \cite{OWY17}.   It is interesting to notice that  if $R$ is an excellent normal local domain of dimension two, then 
$R$ has a rational singularity (resp. minimally elliptic singularity) if and only if $\ov{\e}_2(I)=0$ for every $\m$-primary ideal $I$ in $R$ (resp. $R$ is Gorenstein and $\max\{\ov{\e}_2(I): I \mbox{ is $\m$-primary}\}=1$), see \cite{OWY17}.
%In \cite[Theorem 3.1]{MOR17} the authors explored the equality $\ov{\e}_1(I) -\ov{\e}_0(I)+\ell_R(R/\ov{I})  =  \ell_R(\ov{I^2}/J\ov{I})+1.$  In this case they obtained the structure of the Sally module which, in particular, implies that $\depth~\ov{G}(I) \geq d-1$.  % that is $\ov{G}(I)$ is almost Cohen-Macaulay. 
%The corresponding equality for integrally closed ideals was studied in \cite{OR16}.

In this paper we present the structure of the Sally module in the case the second normal coefficient  is almost minimal, that is  $ \ov{\e}_2(I) = \ov{\e}_1(I) -\ov{\e}_0(I)+\ell_R(R/\ov{I})+1.$  If in addition $\ov{\e}_3(I) \neq 0,$ we prove  that     $\depth~\ov{G}(I) \geq d-1$ and  $\ov{\rm{r}}_J(I)=3,  $ see Theorem \ref{maintheorem1}.  Example \ref{example} shows that  the result is sharp.  If $ \ov{\e}_3(I) = 0 $ and $d=3,$ we prove that $\ov{G}(I^{\ell}) $ is Cohen-Macaulay for all  $\ell \ge 2$ (see Theorem \ref{d=3}). Actually, if  $R$ is Gorenstein, S. Itoh in \cite{Itoh92} conjectured that if $ \ov{\e}_3(I) = 0,  $ then $\ov{G}(I )$ is Cohen-Macaulay. Recently the conjecture was studied by several authors, see for instance \cite{CPR16} and  \cite{KM15}, but as far as we know it is still open.  In this paper a deep analysis of the case $  \ov{\e}_3(I)=0$ has been presented, see Theorem \ref{d=3}.  Our hope is that these results will be successfully applied to give 
new insights to prove (or disprove) the long-standing conjecture by Itoh.  An interesting analysis on the vanishing of  the higher normal coefficients is presented in \cite{GMV19}.

The main tools that we use in this paper are the study of the vanishing of local cohomology modules of the normalized Rees algebra, see Theorem \ref{Itoh2} and Lemma \ref{cohomology} and   the Sally module introduced by  W. V. Vasconcelos in \cite{V94}.  In particular we study the structure of  a suitable filtration of the Sally module considered by M. Vaz Pinto in \cite{VP96}, see Definition \ref{VP} and  Proposition \ref{C_2}, useful for getting information on the reduction number. 

 % An interesting analysis on the higher normal coefficients is presented in \cite{GMV19}.
%We remark that it is very rare to obtain the exact value of reduction number even if the Sally module is Cohen-Macaulay. In this sense our Theorem \ref{main} is more rigid and is not a straightforward generalization of \cite[Theorem 3.1]{MOR17}.
\vskip 2mm
\vskip 2mm
% It is a mystery that the normal Hilbert coefficients and $\ov{G}(I)$ have very nice properties as compared to the Hilbert coefficients of an ideal. Our result continues to support this enigma.

% Thus the ideals $I$ with $\overline{\e}_1(I)=\ov{\e}_0(I)-\ell_R(R/\ov{I})+\ell_R(\ov{I^2}/J \ov{I})+1$ enjoy nice properties and it seems natural to ask when the equality  $\ov{\e}_2(I) = \ov{\e}_1(I) -\ov{\e}_0(I)+\ell_R(R/\ov{I})+1$ holds true. 
% Then the main result of this paper is stated as follows. 

%The main tool that we use in this paper is an analogue of the Sally module introduced by W. Vasconcelos  \cite{V94} for integral closure filtration and a filtration of the Sally module introduced by M. Vaz Pinto in \cite{VP96}. In \cite{CPR16} authors used the normal Sally module to study the equality $\bar{\e}_1(I)=\bar{\e}_0(I)-\ell_R(R/\ov{I})+1.$ In this paper we use $C^{(2)}$ (see Section 2). 
%Hence we develop some basics and some important results on $C^{(2)},$  some of them are stated in a general setting. Our hope is that these will be successfully applied to give new insight to problems related to the normal Hilbert coefficients. We prove our main theorem (Theorem \ref{main}) in Section 3. In Section 4 we derive few consequences of Theorem \ref{main}. 
%In particular we recover the result of A. Corso, C. Polini and M. E. Rossi \cite[Theorem 2.5]{CPR16}. 

%\vskip 3mm
\section*{Acknowledgements}
We thank Manoj Kummini and Claudia Polini for helpful conversations throughout the preparation of the manuscript. 

\section{Notation and Preliminaries}
In \cite{MOR17} the first three authors   proved useful  properties of the Sally module associated to any $I$-admissible filtration that  we will need to prove our main result. In this section we recall part of these results. %summarize the basic properties  which we will be needed to prove our main result. 
We refer to \cite{OR16} and to \cite{MOR17} for  details.

Throughout this paper, let $(R,\m)$ be a Cohen-Macaulay local ring with infinite residue field and $I$ an $\m$-primary ideal in $R.$ 
Recall that a {\it a filtration} of ideals $\mathcal{I}:=\{I_n\}_{n \in \mathbb{Z}}$ is a chain of ideals of $R$ such that $R=I_0$ and $I_n \supseteq I_{n+1}$ for all $n \in \mathbb{Z}$.
We say that a filtration $\mathcal{I}$ is {\it $I$-admissible} if for all $m,n \in \mathbb{Z},$ $I_m \cdot I_n \subseteq I_{m+n}, ~I^n \subseteq I_n$ and there exists $k \in \mathbb{N}$ such that $I_{n} \subseteq I^{n-k}$ for all $n \in \mathbb{Z}.$ If $R$ is analytically unramified, then $\{\ov{I^n}\}_{n \in \mathbb{Z}}$ is an $I$-admissible filtration. In fact, by a classical result of Rees \cite{Ree61}, $R$ is analytically unramified if and only if the  filtration $\{\ov{I^n}\}_{n \in Z}$ is $I$-admissible for some (equiv. for all) $\m$-primary ideal $I$ in $R.$

\vskip 2mm

For an $I$-admissible filtration $\mathcal I=\{I_n\}_{n \in \mathbb{Z}},$ let 
$$
\R(\mathcal I)=\oplus_{n\geq 0}I_nt^n \subseteq R[t], \ \ \ \R'(\mathcal I)=\oplus_{n \in \mathbb{Z}}I_nt^n \subseteq R[t,t^{-1}], \ \ \mbox{and} \ \ G(\mathcal I)=\R'(\mathcal I)/t^{-1}\R'(\mathcal I) 
$$
denote, respectively, the Rees algebra, the extended Rees algebra, and the associated graded ring of $\mathcal{I}$ where $t$ is an indeterminate over $R$.  
If $\mathcal{I}=\{I^n\}_{n\in \mathbb{Z}}$ (resp. $\{\ov{I^n}\}_{n \in \mathbb{Z},}$), we write $\mathcal{R}(I),~\mathcal{R'}(I)$ and $G(I)$ (resp. $\ov{\mathcal{R}}(I),~\ov{\mathcal{R'}}(I)$ and $\ov{G}(I)$)
%We set
%$$\ov{\mathcal{R}}(I) := \oplus_{n \geq 0}\ov{I^n}t^n  \ \subseteq R[t], \ \ ~~~~\ov{\mathcal{R'}}(I) := \oplus_{n \in \mathbb{Z}}\ov{I^n}t^n \ \subseteq R[t,t^{-1}], \ \
%\operatorname{ and } 
%\ov{G}(I):= \ov{\mathcal{R'}}(I)/t^{-1}\ov{\mathcal{R'}}(I) 
% \cong \bigoplus_{n \geq 0}\ov{I^n}/\ov{I^{n+1}}
%$$
for the the Rees algebra, the extended Rees algebra, and the associated graded ring of $\mathcal{I},$  respectively. 
%for the integral closure of $\mathcal{R}(I):=\oplus_{ n \geq 0} I^nt^n$ and $\mathcal{R'}(I):=\oplus_{ n \geq 0} I^nt^n$ in $R[t]$ and $R[t,t^{-1}]$, and the associated graded ring of $\{\ov{I^n}\}_{n \in \Z}$ respectively.
We set 
$ T=\mathcal{R}(J)=\mathcal{R}(\{J^n\}_{n \in \mathbb{Z}})$ where $J$ is a minimal reduction of $I.$
%and $M=\m T+T_+$ denotes the graded maximal ideal of $T$.
Then $\mathcal{R}(\mathcal{I})$ is a module finite extension of $T. $ Hence  $ \ell_R(R/{I_{n+1}})$ for large  $n$ is a polynomial and we denote  by 
  ${\e}_i(\mathcal I)$ the {\it   Hilbert coefficients} of ${\mathcal I}. $  
\vskip 2mm

Following Vasconcelos \cite{V94}, we consider 
$$
S_J(\mathcal{I}):= \frac{\mathcal{R}(\mathcal{I})_{\geq 1} t^{-1}}{I_1 T}\cong \oplus_{n \geq 1}I_{n+1}/J^nI_1
$$ 
the {\it Sally module} of $\mathcal{I}$ with respect to $J$. Notice that $S_J(\mathcal{I})$ is a finite $T$-module. In \cite{VP96} Pinto introduced a filtration of the Sally module in the case $\mathcal{I}=\{I^n\}_{n \in \Z}$.  She constructed the following graded modules to decompose the structure of Sally modules.

\begin{defn}\label{VP}
For each $\ell \geq 1$, consider the graded $T$-module
$$
C^{(\ell)}:=C^{(\ell)}_J(\mathcal{I})= \frac{\mathcal{R}(\mathcal{I})_{\geq \ell} t^{-1}}{I_{\ell} Tt^{\ell-1}} \cong \oplus_{n \geq \ell} {I_{n+1}}/J^{n-\ell+1}I_{\ell}. 
$$
Let $L^{(\ell)}:=L^{(\ell)}_J(\mathcal{I})= [C^{(\ell)}]_{\ell} T$ be the $T$-submodule of $C^{(\ell)}$.
Then 
$L^{(\ell)} \cong \bigoplus_{n \geq \ell} J^{n-\ell}I_{\ell+1}/J^{n-\ell+1}I_{\ell}$. 
% and $C^{(\ell)}/L^{(\ell)} \cong C^{(\ell+1)}$ as graded $T$-modules, 
Hence for every $\ell \ge 1$  we have the following natural exact sequence of graded $T$-modules
$$ 0 \to L^{(\ell)} \to C^{(\ell)} \to C^{(\ell+1)} \to 0. $$
\end{defn}
%Throughout this section we set 
%\[
% C^{(\ell)}:=C^{(\ell)}_J(\mathcal{I}) \mbox{ and }~L^{(\ell)}:=L^{(\ell)}_J(\mathcal{I}) 
%\]
%unless otherwise specified. 
Notice that $C^{(1)}=S,$ and since $\mathcal{R}(\mathcal I)$ is a finite graded $T$-module, $C^{(\ell)}$ and $L^{(\ell)}$ are finitely generated graded $T$-modules for every $\ell \geq 1.$

\vskip 2mm

In this paper, the structure of the graded module $C^{(2)}$ plays a fundamental role.
The following lemma was proved for the $I$-adic filtration in \cite[Lemma 2.1]{OR16}, but the same proof works for any $I$-admissible filtration. 

\begin{lem}\label{fact}
With the notations as above we have
\begin{itemize}
\item[$(1)$] $\m C^{(2)}=(0)$ if and only if $\m I_{n+1} \subseteq J^{n-1}I_2$ for all $n \geq 2$.
\item[$(2)$] $C^{(2)}=(0)$ if and only if $I_{n+1}=JI_n$ for all $n \geq 2$, and 
\item[$(3)$] $C^{(2)}=[C^{(2)}]_2 T$ if and only if $I_{n+1}=JI_n$ for all $n \geq 3$.
\end{itemize}
\end{lem}
 
In the following proposition we need to assume  $J \cap I_2=JI_1  $  where $J$ is a minimal reduction of $\mathcal{I}. $
This condition is automatically satisfied if $\mathcal{I}=\{\m^n\}_{n \in \Z}$ or if $\mathcal{I}=\{\overline{I^n}\}_{n \in \mathbb{Z}}$ (see \cite{H87, Itoh88}).
%We also notice that, under the condition $J \cap I_2=JI_1$, $\ell_R(I_2/JI_1)=\e_0(\mathcal{I})+(d-1)\ell_R(R/I_1)-\ell_R(I_1/I_2)$
%holds true (see for instance \cite{RV10}), so that  $\ell_R(I_2/JI_1)$ does not depend on a minimal reduction $J$ of $\mathcal{I}$. 
%We remark that the following proposition was proved in \cite[Propositions 2.2, 2.8, and 2.9]{OR16} in the case $\mathcal{I}=\{I^n\}_{n \in \Z}.$ 
% For the sake of completeness we give a proof for any $I$-admissible filtration.
We set $\mathrm{HS}_{G(\mathcal{I})}(z)$ and $\mathrm{HS}_{C^{(2)}}(z)$ denote the Hilbert series of $G(\mathcal I)$ and $C^{(2)}$ respectively.

\begin{prop}$($\cite[Proposition 2.4]{MOR17}$)$\label{C_2}  
Let $\p=\m T$ and suppose that $J \cap I_2=JI_1$. Then the following assertions hold true.
\begin{enumerate}
\item \label{C_2(1)} $\Ass_T C^{(2)} \subseteq \{\p\}$. Hence $\dim_TC^{(2)}=d$, if $C^{(2)} \neq (0)$.
\item \label{C_2(2)} For all $n \geq 0,$
\begin{eqnarray*}
\ell_R(R/I_{n+1})&=&\e_0(\mathcal{I})\binom{n+d}{d}-\{\e_0(\mathcal{I})-\ell_R(R/I_1)+\ell_R(I_2/JI_1)\}\binom{n+d-1}{d-1}\\
&& + \ell_R(I_2/JI_1)\binom{n+d-2}{d-2}-\ell_R([C^{(2)}]_n).
\end{eqnarray*}
\item \label{C_2(3)} $\e_1(\mathcal{I})=\e_0(\mathcal{I})-\ell_R(R/I_1)+\ell_R(I_2/JI_1)+\ell_{T_{\p}}(C^{(2)}_{\p}).$
\item  \label{C_2(4)} $\mathrm{HS}_{G(\mathcal{I})}(z)=\frac{\ell_R(R/I_1)+(\e_0(\mathcal{I})-\ell_R(R/I_1)-\ell_R(I_2/JI_1))z+\ell_R(I_2/JI_1)z^2}{(1-z)^d}-(1-z) \mathrm{HS}_{C^{(2)}}(z)$, 
\item \label{C_2(5)} Suppose $C^{(2)} \neq (0)$. 
Then $\depth ~G(\mathcal{I}) =\depth_TC^{(2)}-1$, if $\depth_TC^{(2)} < d$. Moreover, $\depth~ G(\mathcal{I}) \geq d-1$ if and only if $C^{(2)}$ is a Cohen-Macaulay $T$-module.
\end{enumerate}
\end{prop}

We recall that, by using  \cite[Proposition 2.11]{MOR17} and  \cite{Phu18}, we have the following interesting result.

\begin{prop}\label{S2}  
Let $d \geq 2$. Then the graded $T$-module ${C}^{(2)}_J(\{\ov{I^n}\}_{n \in \Z})$ satisfies the Serre's property $(S_2)$ as a $T/\Ann_T \,( {C}^{(2)}_J(\{\ov{I^n}\}_{n \in \Z}))$-module.
\end{prop}

\section{The structure of the Sally module when $\ov{\e}_3(I) \neq 0$} 
In this section we prove the first main result of this paper (Theorem \ref{maintheorem1}). 
We set $\overline{C}=\ov{C}_J(I)=C^{(2)}_J(\{\overline{I^n}\}_{n \in \Z})$. 
In the following theorem we recall few results on the vanishing of local cohomology modules from \cite{Itoh92} (see also \cite{HU14}). 
From now onwards we set $M'=(t^{-1},\m, It){\mathcal R'}(I)$ for the graded maximal ideal of ${\mathcal R'}(I)$ and $N'=It{\mathcal R'}(I)$.

\begin{thm}{$($\cite[Proposition 13]{Itoh92}$)$}\label{Itoh2}
Suppose that $d \geq 2$. 
Then we have the following. 
\begin{enumerate}
\item $[\H_{N'}^i(\ov{\mathcal R'}(I))]_n=(0)$ for all $ n \gg 0$ and all $i \geq 0$;
\item \label{Itoh2(2)} $\H_{M'}^0(\ov{\mathcal R'}(I))=\H_{M'}^1(\ov{\mathcal R'}(I))=(0)$;
\item \label{Itoh2(3)} $[\H_{M'}^2(\ov{\mathcal R'}(I))]_n=(0)$ for $n \leq 0$;
\item $\H_{M'}^i(\ov{\mathcal R'}(I)) \cong \H_{N'}^i(\ov{\mathcal R'}(I))$ for $0 \leq i \leq d-1$.
\end{enumerate}
\end{thm}

To prove the main result of this section we use induction on the dimension $d.$ 
One of the main difficulties in applying the induction on $d$ for the normal filtration is that the image of a normal ideal going modulo a superficial element need not be normal.  Thanks to  \cite[Theorem 1]{Itoh92} (see also  \cite{HU14}) we may choose $a_1 \in I $ such that $\ov{I(R/(a_1))}=\ov{I}(R/(a_1))$, and $\ov{I^{n}(R/(a_1))}=\ov{I^n}(R/(a_1))$ for all $n \gg 0$. In particular $a_1t$ is $\ov{G}(I)$-regular. 
From now onwards we set $S=R/(a_1)$.  We prove  the following important  lemma which shows that one of the main difficulties in the study of the normal Hilbert coefficients is that in general 
$\ov{I^nS} \neq \ov{I^n} S$ for $n \in \Z.$

\begin{lem}\label{cohomology}
Suppose that $d \geq 3$.
Then 
% we have the following where $M'=(\m, T_+,t^{-1}){\mathcal R'}(J)$ 
\begin{itemize}
\item[$(1)$] $\H_{M'}^1({\mathcal R'}(\{\ov{I^n}S\}_{n \in \Z}))_n \cong \ov{I^nS}/\ov{I^n}S$ for all $n \in \Z$, and
\item[$(2)$] $\H_{M'}^i(\ov{\mathcal R'}(IS)) \cong \H_{M'}^i({\mathcal R'}(\{\ov{I^n}S\}_{n \in \Z}))$ for all $i \geq 2$.
\end{itemize}
\end{lem}

\begin{proof}
Consider the canonical exact sequence 
$$ 0 \to \mathcal{R'}(\{\ov{I^n}S\}_{n \in \Z}) \to \ov{\mathcal{R'}}(IS) \to \ov{\mathcal{R'}}(IS)/\mathcal{R'}(\{\ov{I^n}S\}_{n \in \Z}) \to 0 $$
of graded $T$-modules. 
Then, since $\ov{I^nS}=\ov{I^n} S$ for all $n \leq 0$ and, by Theorem \cite[Theorem 1]{Itoh92},  for all $n \gg 0$, the module $\ov{\mathcal{R'}}(IS)/\mathcal{R'}(\{\ov{I^n}S\}_{n \in \Z})\cong \bigoplus_{n \in \Z} \ov{I^nS}/\ov{I^n} S$ is finitely graded. 
% Hence $[\ov{\mathcal{R'}}(IR')/\mathcal{R'}(\{\ov{I^n}R'\}_{n \in \Z})]_n \cong \ov{I^nR'}/\ov{I^n} R'$ for $n \in \Z$ and by Theorem \ref{Itoh1}\eqref{Itoh1(2)} $\ov{I^nR'}=\ov{I^n} R'$ for all $n \gg 0$ and 
Therefore taking the local cohomology functor $\H_{M'}^i(*)$ to the above exact sequence and using Theorem \ref{Itoh2}\eqref{Itoh2(2)}, we get $\H^1_{M'}({\mathcal R'}(\{\ov{I^n}S\}_{n \in \Z})) \cong \ov{\mathcal{R'}}(IS)/\mathcal{R'}(\{\ov{I^n}S\}_{n \in \Z})$ and $\H^i_{M'}({\mathcal R'}(\{\ov{I^n}S\}_{n \in \Z}))) \cong \H^i_{M'}(\ov{\mathcal{R'}}(IS))$ for $i \geq 2$ as required.
\end{proof}

We remark that the following result works like  the Sally's machine  \cite[Lemma 1.4]{RV10}, but  is not a consequence of it. 

% We have the following result which corresponds to the {\it Sally's machine} for the associate graded rings of normal filtrations.

\begin{prop}\label{sally}
Assume  $d \geq 3$ and   $\depth~ \ov{\rmG}(IS) \geq 2$.
Then we have $\depth ~\ov{\rmG}(I) =\depth ~\ov{\rmG}(IS)+1. $ 
\end{prop}

\begin{proof}
Let us look at the exact sequence
$$ 0 \to \H_{M'}^1({\mathcal R'}(\{\ov{I^n}S\}_{n \in \Z})) \to \H_{M'}^2(\ov{\mathcal R'}(I))(-1) \to \H_{M'}^2(\ov{\mathcal R'}(I)) \to \H_{M'}^2({\mathcal R'}(\{\ov{I^n}S\}_{n \in \Z})) \to $$
of local cohomology modules induced by the canonical exact sequence
$$ 0 \longrightarrow \ov{\mathcal R'}(I)(-1) \overset{a_1t}{\longrightarrow} \ov{\mathcal R'}(I) \longrightarrow {\mathcal R'}(\{\ov{I^n}S\}_{n \in \Z}) \longrightarrow 0.$$
Because $\depth\ \ov{\rmG}(IS) \geq 2$, we have $\depth\ \ov{\mathcal R'}(IS) \geq 3$ so that $\H_{M'}^2({\mathcal R'}(\{\ov{I^n}S\}_{n \in \Z}) \cong \H_{M'}^2(\ov{\mathcal R'}(IS))=(0)$ by Lemma \ref{cohomology}$(2)$.
This gives an epimorphism $ \H_{M'}^2(\ov{\mathcal R'}(I))(-1) \to \H_{M'}^2(\ov{\mathcal R'}(I)) \to 0 $ of graded $T$-modules.
Then since $\H_{M'}^2(\ov{\mathcal R'}(I))$ is finitely graded by Theorem \ref{Itoh2}, we get $\H_{M'}^2(\ov{\mathcal R'}(I))=(0)$ so that $\H_{M'}^1({\mathcal R'}(\{\ov{I^n}S\}_{n \in \Z}))=(0)$ by the above exact sequence.
Then because $\ov{I^n S}=\ov{I^n}S$ for all $n \in \Z$ by Lemma \ref{cohomology}$(1)$, we have $\ov{\mathcal R'}(IS) = {\mathcal R'}(\{\ov{I^n}S\}_{n \in \Z})$.
Thus, we get $\depth\ \ov{\rmG}(I) = \depth\ \ov{\mathcal R'}(I)-1= \depth\ {\mathcal R'}(\{\ov{I^n}S\}_{n \in \Z}) = \depth\ \ov{\mathcal R'}(IS) = \depth\ \ov{\rmG}(IS)+1$ as required.
\end{proof}

The following lemma is a consequence of the Grothendieck-Serre formula \cite[Theorem 4.1]{Bla97}.  % plays a crucial role in the study of the normal Hilbert coefficients.  

\begin{lem}\label{rem}
%We have the following.
\begin{asparaenum}
 \item[$(1)$] \label{rem1} Suppose $d \geq 3$. Then, for each $n \in \Z$, we have
{\small {$$ \sum_{i=0}^d(-1)^i \ov{\e}_i(I)\binom{n+d-i}{d-i}-\ell_R(R/\ov{I^{n+1}})=\sum_{k \geq n+2}\ell_R(\ov{I^kS}/\ov{I^k}S)-\sum_{i=2}^{d-1}(-1)^i\ell_R(\H_{N'}^i(\ov{\mathcal R'}(IS))_{\geq n+2}).$$}}
\item[$(2)$] \label{rem2} Suppose $d=3.$ Then 
\[
\ov{\e}_3(I)= \ov{\e}_2(I)-\ov{\e}_1(I)+\ov{\e}_0(I)-\ell_R(R/\ov{I})-\sum_{n \geq 2}\ell_R(\ov{I^nS}/\ov{I^n}S)+ \ell_R (\H_{N'}^2(\ov{\mathcal R'}(IS))_{\geq 2}). 
 \] 
\end{asparaenum}
\end{lem}

\begin{proof}
%\begin{asparaenum}
(1) Consider the short exact sequence
$$ 0 \longrightarrow \ov{\mathcal R'}(I)(-1) \overset{a_1t}{\longrightarrow} \ov{\mathcal R'}(I) \longrightarrow {\mathcal R'}(\{\ov{I^n}S\}_{n \in \Z}) \longrightarrow 0.$$
By using Theorem \ref{Itoh2} we get a long exact sequence of local cohomology modules
% Let us consider the long exact sequence
% $$ 
% 0 \to \H_{N'}^1({\mathcal R'}(\{\ov{I^n}S\}_{n \in \Z})) \to \H_{N'}^2(\ov{\mathcal R'}(I))(-1) \to \H_{N'}^2(\ov{\mathcal R'}(I)) \to \H_{N'}^2({\mathcal R'}(\{\ov{I^n}S\}_{n \in \Z})) $$
$$ \cdots \to \H_{N'}^{i-1}({\mathcal R'}(\{\ov{I^n}S\}_{n \in \Z})) \to \H_{N'}^i(\ov{\mathcal R'}(I))(-1) \to \H_{N'}^i(\ov{\mathcal R'}(I)) \to \H_{N'}^i({\mathcal R'}(\{\ov{I^n}S\}_{n \in \Z})) \to $$
% $$ \cdots \to \H_{N'}^{d-1}({\mathcal R'}(\{\ov{I^n}S\}_{n \in \Z})) \to \H_{N'}^d(\ov{\mathcal R'}(I))(-1) \to \H_{N'}^d(\ov{\mathcal R'}(I)) \to 0.$$
% of local cohomology modules induced by the exact sequence
% $$ 0 \to \ov{\mathcal R'}(I)(-1) \overset{a_1t}{\to} \ov{\mathcal R'}(I) \to {\mathcal R'}(\{\ov{I^n}R'\}_{n \in \Z}) \to 0$$
Hence we have 
$$ \sum_{i=2}^d(-1)^i\{\ell_R(\H_{N'}^i(\ov{\mathcal R'}(I))_{ \geq n+1})-\ell_R(\H_{N'}^i(\ov{\mathcal R'}(I))_{\geq n+2}) -\ell_R(\H_{N'}^{i-1}({\mathcal R'}(\{\ov{I^n}S\}))_{\geq n+2})\}=0 $$
for all $n \in \Z$.
Thanks to \cite[Theorem 4.1]{Bla97}, for each $n \in \Z$, we have
\begin{eqnarray*}
&& \sum_{i=0}^d(-1)^i\ov{\e}_i(I)\binom{n+d-i}{d-i}-\ell_R(R/\ov{I^{n+1}}) = \sum_{i=2}^d(-1)^i\ell_R(\H_{N'}^i(\ov{\mathcal R'}(I))_{n+1}) \\
&=& \sum_{i=2}^d(-1)^i \{ \ell_R(\H_{N'}^i(\ov{\mathcal R'}(I))_{ \geq n+1})-\ell_R(\H_{N'}^i(\ov{\mathcal R'}(I))_{\geq n+2})\}\\
&=& \sum_{i=2}^d(-1)^i  \ell_R(\H_{N'}^{i-1}({\mathcal R'}(\{\ov{I^n}S\}))_{\geq n+2}) 
=\sum_{k \geq n+2}\ell_R(\ov{I^kS}/\ov{I^k}S)-\sum_{i=2}^{d-1}(-1)^i\ell_R(\H_{N'}^i(\ov{\mathcal R'}(IS))_{\geq n+2})
\end{eqnarray*}
%$ \sum_{i=0}^d(-1)^i\ov{\e}_i(I)\binom{n+d-i}{d-i}-\ell_R(R/\ov{I^{n+1}})  
% = \sum_{i=2}^d(-1)^i\ell_R(\H_{N'}^i(\ov{\mathcal R'}(I))_{n+1}) = 
%  \sum_{i=2}^d(-1)^i \{ \ell_R(\H_{N'}^i(\ov{\mathcal R'}(I))_{ \geq n+1})-\ell_R(\H_{N'}^i(\ov{\mathcal R'}(I))_{\geq n+2})\} 
%=  \sum_{i=2}^d(-1)^i  \ell_R(\H_{N'}^{i-1}({\mathcal R'}(\{\ov{I^n}S\}))_{\geq n+2})  
%=  \sum_{k \geq n+2}\ell_R(\ov{I^kS}/\ov{I^k}S)-\sum_{i=2}^{d-1}(-1)^i\ell_R(\H_{N'}^i(\ov{\mathcal R'}(IS))_{\geq n+2})
% $
because $\H_{N'}^1(\ov{\mathcal R'}(\{ \ov{I^n}S \}))_k \cong \ov{I^kS}/\ov{I^k}S$ for all $k \in \Z$ and $\H_{N'}^i({\mathcal R'}(\{ \ov{I^n}S \})) \cong \H_{N'}^i(\ov{\mathcal R'}(IS)) $ as graded $T$-modules for $i \geq 2$ by Lemma \ref{cohomology}. (2) Follows from (1).
%\end{asparaenum}  
\end{proof}

% \begin{lem}
%  \label{lemma:e2almostminimal}
%  Suppose $d=3$ and $\ov{\e}_2(I)=\ov{\e}_1(I)-\ov{\e}_0(I)+\ell_R(R/\ov{I})+1.$ Then 
%  \[
% \ov{\e}_3(I)= 1-\sum_{n \geq 2}\ell_R(\ov{I^nS}/\ov{I^n}S)  
%  \]
% \end{lem}
% \begin{proof}
%  
% \end{proof}

The following result plays a key role for our proof of Theorem \ref{maintheorem1}.

\begin{thm}\label{key}
%Let $(R,\m)$ be an analytically unramified Cohen-Macaulay local ring of dimension $d \geq 2$ and $I$ an $\m$-primary ideal in $R$. 
Suppose  $d \geq 2$.
Assume  $\ov{\e}_2(I)=\ov{\e}_1(I)-\ov{\e}_0(I)+\ell_R(R/\ov{I})+1$ and $\ov{\e}_3(I) \neq 0$ (if $d \geq 3$), then  $\ell_R(\ov{I^3}/J\ov{I^2})=1$ and $\ov{I^{n+1}}=J\ov{I^n}$   for all $n \geq 3$.
\end{thm}

\begin{proof}
We prove the required equalities by induction on $d$. 
Suppose that $d=2$.
Then since $\depth ~\ov{G}(I) > 0,$ using \cite[Proposition 4.6]{HM97} we get $\ov{\e}_1(I)=\sum_{n \geq 0}\ell_R(\ov{I^{n+1}}/J\ov{I^n})$ and $\ov{\e}_2(I)=\sum_{n \geq 1}n \ell_R(\ov{I^{n+1}}/J\ov{I^n})$.
Therefore
$$1= \ov{\e}_2(I)-\ov{\e}_1(I)+\ov{\e}_0(I)-\ell_R(R/\ov{I})=\sum_{n \geq 2}(n-1)\ell_R(\ov{I^{n+1}}/J\ov{I^n}).$$
This implies that $\ell_R(\ov{I^{3}}/J \ov{I^2})=1$ and $\ov{I^{n+1}}=J \ov{I^n}$ for all $n \geq 3$ as required.

Suppose that $d \geq 3$ and that our assertion holds true for $d-1$.
We then have $\ov{\e}_2(IS)=\ov{\e}_1(IS)-\ov{\e}_0(IS)+\ell_{S}(S/\ov{IS})+1$, and $\ov{\e}_3(IS) \neq 0$ if $d \geq 4$, by \cite[Theorem 1]{Itoh92}.
The hypothesis of induction on $d$ implies that we have $\ell_R(\ov{I^3S}/J\ov{I^2S})=1$ and $\ov{I^{n+1}S}=J\ov{I^nS}$ for all $n \geq 3$. Hence to prove our assertion it is enough to show that $\ov{I^{n}S}=\ov{I^n}S$ holds true for all $n \in \Z$.

Suppose $d \geq 4$. 
By using \cite[Theorem 3.1]{MOR17}, we have $\depth ~{\ov{\rmG}}(IS) \geq (d-1)-1=d-2 \geq 2.$ 
Then $\depth~ \ov{G}(I) \geq 3$ by Proposition \ref{sally}.
Hence $\depth ~{\mathcal R'}(\{\ov{I^n}S\}_{n \in \Z})=\depth ~ \ov{\mathcal{R'}}(I)-1 \geq 3$.
Therefore $\ov{I^nS}=\ov{I^n}S$ for all $n \in \Z$ by Lemma \ref{cohomology} as required.

\vskip 3mm

%%%%%%%%%%%%%%%%%%%%%%%%%%%         d=3           %%%%%%%%%%%%%%%%%%%%%%%%%%%%%%

Suppose that $d=3$.
Consider the exact sequence
$$\cdots \longrightarrow \H_{N'}^2(\ov{\mathcal{R'}}(IS))(1) \overset{t^{-1}}{\longrightarrow} \H_{N'}^2(\ov{\mathcal{R'}}(IS)) \longrightarrow \H_{N'}^2(\ov{G}(IS)) \longrightarrow 0 $$
of local cohomology modules which is induced by the canonical exact sequence
$$ 0 \longrightarrow \ov{\mathcal{R'}}(IS)(1) \overset{t^{-1}}{\longrightarrow} \ov{\mathcal{R'}}(IS) \longrightarrow \ov{G}(IS) \longrightarrow 0.$$
Since $\ov{I^{n+1} S}=J\ov{I^n S}$ for all $n \geq 3$, we get ${\rm{a}_2}(\ov{G}(IS))+2 \leq \ov{\rm{r}}_{JS}(IS) \leq 3$ by \cite[Proposition 3.2]{HZ94} $($also \cite[Proposition 3.2]{Tru87}$)$ so that $[\H_{N'}^2(\ov{G}(IS))]_n=(0)$ for all $n \geq 2$ where ${\rm{a}_2}(\ov{G}(IS)):=\sup\{n \in \Z \ | \ [\H_{N'}^2(\ov{G}(IS))]_n \neq (0) \}$ denotes the ${\rm{a}}$-invariant of $\ov{G}(IS)$.
Hence $\H_{N'}^2(\ov{\mathcal{R'}}(IS))_n=(0)$ for all $n \geq 2$.
Therefore we have
\begin{eqnarray*}
\ov{\e}_3(I) &=& \ov{\e}_0(I)-\ov{\e}_1(I)+\ov{\e}_2(I)-\ell_R(R/\ov{I})-\sum_{n \geq 2}\ell_R(\ov{I^nS}/\ov{I^n}S) \nonumber
= 1-\sum_{n \geq 2}\ell_R(\ov{I^nS}/\ov{I^n}S) \label{Eqn:e3}.
\end{eqnarray*}
by Lemma \ref{rem}(2).
Therefore, because $\ov{\e}_3(I) >0$ by our assumption and \cite[Theorem 3]{Itoh92}, we get $\ov{I^nS}=\ov{I^n}S$ for all $n \in \Z$ as required.
\end{proof}

Let $B:=T/\m T \cong (R/\m)[X_1,X_2,\cdots,X_d]$ the polynomial ring with $d$ indeterminates over the field $R/\m$.

Now we give a complete structure of the Sally module and we describe the Hilbert series of the associated graded ring in the case $\overline{\e}_2(I)=\ov{\e}_1(I)-\ov{\e}_0(I)+\ell_R(R/\ov{I})+1$ and $\ov{\e}_3(I) \neq 0.$ 
% In particular, we prove that $\depth ~\ov{G}(I) \geq d-1$ and $\ov{\rm{r}}_J(I) \leq 2.$ 
% now in a position to give a main theorem of this section. \textcolor{red}{write}

\begin{thm}\label{maintheorem1}
%Let $(R,\m)$ be an analytically unramified Cohen-Macaulay local ring of dimension $d \geq 2$ and $I$ an $\m$-primary ideal in $R$. 
Suppose that $d \geq 2$.
Then following statements are equivalent:
\begin{enumerate}
\item[$(1)$] $\overline{\e}_2(I)=\ov{\e}_1(I)-\ov{\e}_0(I)+\ell_R(R/\ov{I})+1$ and, if $d \geq 3, $ $\ov{\e}_3(I) \neq 0$,
\item[$(2)$] $\ov{\e}_2(I)=\ell_R(\ov{I^2}/J\ov{I})+2$,
\item[$(3)$] $\ov{C}_J(I) \cong B(-2)$ as graded $T$-modules, and
\item[$(4)$] $\ell_R(\ov{I^{3}}/J \ov{I^2})=1$ and $\ov{I^{n+1}}=J \ov{I^n}$ for all $n \geq 3$. 
 
\end{enumerate}
In this case, the following assertions follow:
\begin{itemize}
\item[$(i)$] $\ov{\e}_1(I)=\ov{\e}_0(I)-\ell_R(R/\ov{I})+\ell_R(\ov{I^2}/J\ov{I})+1$.
\item[$(ii)$] $\ov{\e}_3(I)=1$ if $d \geq 3$, and $\ov{\e}_i(I)=0$ for $4 \leq i \leq d$.
\item[$(iii)$] $HS_{\ov{G}(I)}(z)=\frac{\ell_R(R/\ov{I})+(\e_0(I)-\ell_R(R/\ov{I})-\ell_R(\ov{I^2}/J \ov{I}))z+(\ell_R(\ov{I^2}/J \ov{I})-1)z^2+z^{3}}{(1-z)^d}.$ 
\item[$(iv)$] $\depth ~\ov{G}(I) \geq d-1$, and $\ov{G}(I)$ is Cohen-Macaulay if and only if $\ov{I^3} \nsubseteq J.$ 
\end{itemize}
\end{thm}

\begin{proof}
$(1) \Rightarrow (4)$ follows from Theorem \ref{key}. 
Thanks to \cite[Theorem 3.1]{MOR17}, $(3) \Leftrightarrow (4)$, $(4) \Rightarrow (2)$, and the last assertions $(i) \sim (iv)$ follow. %Hence it suffices to prove $(2) \Rightarrow (1).$ 

\noindent 
$(2) \Rightarrow (1):$ By \cite[Theorem 2]{Itoh92} and the assumption we have the inequalities
$$ \ell_R(\ov{I^2}/J\ov{I})+2 =\ov{\e}_2(I) \geq \ov{\e}_1(I)-\ov{\e}_0(I)+\ell_R(R/\ov{I}) \geq \ell_R(\ov{I^2}/J\ov{I}) $$
and the equality $\ov{\e}_2(I) = \ov{\e}_1(I)-\ov{\e}_0(I)+\ell_R(R/\ov{I})$ is true if and only if the equality $\ov{\e}_1(I)=\ov{\e}_0(I)-\ell_R(R/\ov{I})+\ell_R(\ov{I^2}/J\ov{I})$ is true.
Hence we have $\ov{\e}_1(I)=\ov{\e}_0(I)-\ell_{R}(R/\ov{I})+\ell_{R}(\ov{I^2}/J\ov{I})+1$ and $\ov{\e}_2(I)=\ov{\e}_1(I)-\ov{\e}_0(I)+\ell_R(R/\ov{I})+1$.
Then, thanks to \cite[Theorem 3.1]{MOR17}, we get $\ov{\e}_3(I)=\binom{m}{2} \neq 0$ for some $m \geq 2$. 
%Then, thanks to \cite[Theorem 3.1]{MOR17}, we have $\ov{C} \cong B(-m)$ for some $m \geq 2$ and $\ov{\e}_2(I)=\ell_R(\ov{I^2}/J\ov{I})+m$. Therefore by our assumption $m=2.$ Hence again using \cite[Theorem 3.1]{MOR17}, we get $\ov{\e}_3(I)=1 \neq 0.$ 
% When this is the case, we have
% $$m=\ov{\e}_2(I)-\ell_R(\ov{I^2}/J\ov{I})=\ov{\e}_2(I)-\ov{\e}_1(I)+\ov{\e}_0(I)-\ell_R(R/\ov{I})+1=2.$$
% Thus the result follows.
\end{proof}

By the above result we notice that if 
  $d \geq 3$ and   $\ov{\e}_2(I) \leq \ov{\e}_1(I)-\ov{\e}_0(I)+\ell_R(R/\ov{I})+1, $     then $\ov{\e}_3(I) \leq 1$.
 
\vskip 2mm
%\begin{rem}
%In \cite{Itoh92} Itoh conjectured that if $R$ is Gorenstein of dimension $d \geq 3$ and $\ov{\e}_3(I)=0,$ then $\ov{\rm{r}}(I) \leq 2$. Therefore the assumption $\ov{\e}_3(I) \neq 0$ can be removed in Theorem \ref{maintheorem1} if $R$ is Gorenstein and   Itoh's conjecture is true. In particular 
%  Itoh's conjecture is true for $\ov{I}=\m$ in a Gorenstein local ring  by \cite[Theorem 3(2)]{Itoh92}  and more generally in a Cohen-Macaulay local ring $R$ with $ \type(R)\leq \ell_R(\ov{\m^2}/J \m) + 2$  by  \cite{CPR16},\cite{KM15}.  
%\end{rem} 
 
The next result follows immediately from Theorem \ref{maintheorem1} and \cite[Corollary 4.1]{MOR17}. 
Recall that an ideal $I$ is said to be normal if $\ov{I^n}=I^n$ for all $n \geq 0.$
We set $\e_i(I)=\e_i(\{I^n\}_{n \in \mathbb{Z}})$ denotes the $i$-th Hilbert coefficient of the $I$-adic filtration $\{I^n\}_{n \in \Z}$.

\begin{cor}\label{normal}
Assume $d \geq 2$ and $I$ is normal.
Then the following conditions are equivalent:
\begin{itemize}
 \item[$(1)$] $\e_1(I)=\e_0(I)-\ell_R(R/I)+\ell_R(I^2/JI)+1;$
 \item[$(2)$] $\e_2(I)=\e_1(I)-\e_0(I)+\ell_R(R/I)+1$ and, if $d \geq 3,$ ${\e}_3(I) \neq 0;$
% \item[$(3)$] $\e_2(I)=\ell_R(I^2/JI)+2;$
 \item[$(3)$] ${C^{(2)}_J(\{I^n\}_{n \in \mathbb{Z}})} \cong B(-2)$ as graded $T$-modules;
 \item[$(4)$] $\ell_R(I^{3}/JI^2)=1$ and $I^4=JI^3$.
\end{itemize}
When this is the case, $\depth ~G(I) \geq d-1$, and $G(I)$ is Cohen-Macaulay if and only if $I^3 \nsubseteq J$.
\end{cor}

The following example, due to Huckaba and Huneke \cite[Theorem 3.12]{HH99}, shows  that if $I$ is normal,  $\e_2(I)=\e_1(I)-\e_0(I)+\ell_R(R/I)+1$ and $\e_3(I) \neq 0,$ then $G(I)$ need not be Cohen-Macaulay and hence Theorem \ref{maintheorem1} is sharp. 

\begin{ex}\label{example} 
 Let $K$ be a field of characteristic $\not= 3$ and set $R =
K[\![X,Y,Z]\!]$, where $X,Y,Z$ are indeterminates. Let $N = (X^4,
X(Y^3+Z^3), Y(Y^3+Z^3), Z(Y^3+Z^3))$ and set $I = N + {\mathfrak
m}^5$, where ${\mathfrak m}$ is the maximal ideal of $R$. The
ideal $I$ is a normal ${\mathfrak m}$-primary ideal whose
associated graded  ring ${G}(I)$ has depth $d-1=2. $    Moreover,
$$HS_{{G}(I)}(t)=\frac{31 + 43t +t^2 + t^3}{(1-t)^3},$$
 and hence $\ell_R(R/I)=31,~\e_0(I)=76,~ \e_1(I)=48,~ \e_2(I)=4,~\e_3(I)=1$. Thus $\e_2(I)=\e_1(I)-\e_0(I)+\ell_R(R/I)+1.$ For the computations see \cite[Example 3.2]{CPR05}.
\end{ex}

\section{The structure of the Sally module when $\ov{\e}_3(I) = 0$}

In this section we consider the case $\overline{\e}_2(I)=\ov{\e}_1(I)-\ov{\e}_0(I)+\ell_R(R/\ov{I})+1$ and $\ov{\e}_3(I) = 0$ in three dimensional case. 
% Next theorem is  the main result of this section. 
This case faces the difficult problem stated by Itoh in \cite{Itoh92} on the vanishing of $\ov{\e}_3(I)$ which asserts that  
if $\ov{\e}_3(I)=0$ and $R$ is Gorenstein, then $\ov{G}(I)$ is Cohen-Macaulay or equivalently 
  $\overline{\e}_2(I)=\ov{\e}_1(I)-\ov{\e}_0(I)+\ell_R(R/\ov{I}).$ Hence for the class  of ideals verifying Itoh's conjecture the assumptions of this section doesn't occur. This is the case for instance when $\ov{I}=\m$ and $R$ is Gorenstein,  see\cite[Theorem 3(2)]{Itoh92} (more generally, $R$ satisfying $\ell_R(\ov{I^2}/J\ov{I}) \geq \type(R)-2, $ see \cite{CPR16,KM15}).
If $R$ is not Gorenstein or $R $ is Gorenstein and $\ov{I} \neq \mathfrak m, $ our analysis can be useful for proving or disproving Itoh's conjecture, also because the doubt of the validity of Itoh's conjecture is growing among the experts. 

In Theorem \ref{d=3} we prove that if $\overline{\e}_2(I)=\ov{\e}_1(I)-\ov{\e}_0(I)+\ell_R(R/\ov{I})+1$ and $\ov{\e}_3(I) = 0$, then $\ov{G}(I^\ell)$ is Cohen-Macaulay for all $\ell \geq 2.$  For this purpose we need the following proposition which is a consequence of Serre's formula and it seems to be well known. 
However, for the sake of completeness we give a proof of this. 

We set, for $\ell \in \Z$, ${\mathcal I}^{(\ell)}=\{I_{n\ell}\}_{n \in \Z}$, and ${\rm a}_i({\rm G}({\mathcal I})) = \max\{n \in \Z \ | \ [\H_{M}^i({\rm G}({\mathcal I})]_n \neq (0) \}$ for $i \in \Z$.

\begin{prop}\label{veronese}
Let $\ell > \max\{ {\rm a}_i({\rm G}({\mathcal I})) \, | \, 0 \leq i \leq d \}$ be an integer.
Then we have
$ \ell_R(R/I_{\ell(n+1)})=\sum_{i=0}^d(-1)^i \e_i({\mathcal I}^{(\ell)})\binom{n+d-i}{d-i} $
for all $n \geq 0$.
In particular, the equality
$ \ell_R(R/I_{\ell})=\sum_{i=0}^d(-1)^i \e_i({\mathcal I}^{(\ell)}) $
holds true for all $n \geq 0$.
\end{prop}

\begin{proof}
We have ${\rm a}_i({\rm G}({\mathcal I}^{(\ell)})) \leq \lfloor {\rm a}_i({\rm G}({\mathcal I}))/\ell \rfloor \leq 0 $ by \cite[Theorem 4.2]{HZ94} where $\lfloor q \rfloor =\max\{n \in \Z \ | \ n \leq q \}$ for $q \in \mathbb{Q}$.
Then thanks to Serre's formula
%$$ \ell_R([{\rm G}({\mathcal I}^{(\ell)})]_0)-\sum_{i=0}^{d-1}(-1)^i {\rm e}_i({\mathcal I}^{(\ell)})=\sum_{i=0}^d (-1)^i\ell_R([\H_{M'}^i({\rm G}({\mathcal I}^{(\ell)}))]_0) $$and 
$ \ell_R([{\rm G}({\mathcal I^{(\ell)}})]_k)=\sum_{i=0}^{d-1}(-1)^i\e_i({\mathcal I^{(\ell)}})\binom{k+d-i-1}{d-i-1}$ for all $k \geq 1$.
Hence we have
\begin{eqnarray*}
&& \ell_R(R/I_{\ell (n+1)}) = \sum_{k=0}^{n}\ell_R([{\rm G}({\mathcal I^{(\ell)}})]_k )\\
&=& \sum_{k=0}^n\left\{ \sum_{i=0}^{d-1}(-1)^i\e_i({\mathcal I^{(\ell)}})\binom{k+d-i-1}{d-i-1}+ \ell_R([{\rm G}({\mathcal I^{(\ell)}})]_k)-\sum_{i=0}^{d-1}(-1)^i\e_i({\mathcal I^{(\ell)}})\binom{k+d-i-1}{d-i-1} \right\}\\
%&=& \sum_{k=0}^n\sum_{i=0}^{d-1}(-1)^i\e_i({\mathcal I^{(\ell)}})\binom{k+d-i-1}{d-i-1}+\sum_{k=0}^n\left\{ \ell_R([{\rm G}({\mathcal I^{(\ell)}})]_k)-\sum_{i=0}^{d-1}(-1)^i\e_i({\mathcal I^{(\ell)}})\binom{k+d-i-1}{d-i-1} \right\}\\
&=& \sum_{i=0}^{d-1}(-1)^i\e_i({\mathcal I^{(\ell)}})\binom{n+d-i}{d-i}+\left\{ \ell_R([{\rm G}({\mathcal I^{(\ell)}})]_0)-\sum_{i=0}^{d-1}(-1)^i\e_i({\mathcal I^{(\ell)}}) \right\}
\end{eqnarray*}
for all $n \geq 0$.
Therefore, $(-1)^d{\e_d}({\mathcal I}^{(\ell)})=\ell_R([{\rm G}({\mathcal I^{(\ell)}})]_0)-\sum_{i=0}^{d-1}(-1)^i\e_i({\mathcal I^{(\ell)}})$ and the equality $ \ell_R(R/I_{\ell(n+1)})=\sum_{i=0}^d(-1)^i \e_i({\mathcal I}^{(\ell)})\binom{n+d-i}{d-i} $ holds true for all $n \geq 0$.
\end{proof}

As a consequence of Proposition \ref{veronese} we obtain a result of Rees \cite[Theorem 2.6]{Ree81} (see also  \cite[Theorem 4.5]{H87}) in dimension two which states that: $\ov{\e}_2(I)=0$ if and only if $\ov{\e}_1(I^\ell)=\e_0(I^{\ell})-\ell_R(R/\ov{I^{\ell}})$ for  all $\ell \geq 1$. In particular, by \cite[Theorem 2.9]{RV10} $\ov{\rm G}(I^\ell)$ is Cohen-Macaulay for all $\ell \geq 1$. 
Analogously we obtain the following result on vanishing of $\ov{\e}_3(I)$ in dimension three as a consequence of Proposition \ref{veronese}.
Notice that next result for normal ideals can be also obtained as a consequence of [12, Corollary 5.3.]

\begin{cor}\label{cor3} 
Let $d=3$, 
then the following conditions are equivalent.
\begin{itemize}
\item[$(1)$] $\overline{\e}_3(I)=0$, 
\item[$(2)$] $\overline{\e}_1(I^{\ell})=2\overline{\e}_0(I^{\ell})+\ell_R(R/\overline{I^{\ell}})-\ell_R(\overline{I^{\ell}}/\overline{I^{2\ell}})$ for some (equiv. all) $\ell  >  \max\{ {\rma}_i(\overline{\rmG}(I)) \ | \ 1 \leq i \leq 3 \}$, and
\item[$(3)$] $\overline{\e}_2(I^{\ell})=\overline{\e}_1(I^{\ell})-\overline{\e}_0(I^{\ell})+\ell_R(R/\overline{I^{\ell}})$ for some (equiv. all) $\ell  >  \max\{ {\rma}_i(\overline{\rmG}(I)) \, | \, 1 \leq i \leq 3 \}$.
\end{itemize}
In particular, $\overline{\rm G}(I^{\ell})$ is Cohen-Macaulay  for all $\ell >  \max\{ {\rma}_i(\overline{\rmG}(I)) \ | \ 1 \leq i \leq 3 \}$ if any of the above conditions are satisfied. 
\end{cor}

The last assertion of Corollary \ref{cor3} is a consequence of \cite[Theorem 2(2)]{Itoh92}. Now we are ready to prove the main theorem of this section.

\begin{thm}\label{d=3} 
%Let $(R,\m)$ be an analytically unramified Cohen-Macaulay local ring of dimension three and let $I$ be an $\m$-primary ideal in $R$.
Suppose $d=3$.
Then the following conditions are equivalent.
\begin{itemize}
\item[$(1)$] $\overline{\e}_2(I)=\ov{\e}_1(I)-\ov{\e}_0(I)+\ell_R(R/\ov{I})+1$ and $\ov{\e}_3(I) = 0$;
\item[$(2)$] there exists an exact sequence
$ 0 \to B(-3) \to B(-2)^ {\oplus 3} \to \ov{C} \to 0$
of graded $T$-modules;
\item[$(3)$] there exists an exact sequence
$0 \to B(-2)^{\oplus2} \to \ov{C} \to (B/(a_1t))(-2) \to 0$
of graded $T$-modules.
\end{itemize}
When this is the case, the following assertions hold true:
\begin{itemize}
\item[$(i)$] $\m \overline{C}=(0)$ and $\rank_B\overline{C}=2$, and $\depth_T \ov{C} =2$,
\item[$(ii)$] $\m \ov{I^3} \subseteq J\ov{I^2}$, $\ell_R(\ov{I^3}/J\ov{I^2})=3$, $\ov{I^{n+1}}=J\ov{I^n}$ for all $n \geq 3$,
\item[$(iii)$] $\overline{\e}_1(I)=\ov{\e}_0(I)-\ell_R(R/\ov{I})+\ell_R(\ov{I^2}/J\ov{I})+2$ and $\ov{\e}_2(I)=\ell_R(\ov{I^2}/J\ov{I})+3$,
\item[$(iv)$] ${\mathrm{HS}}_{\ov{\rm G}(I)}(z)=\frac{\ell_R(R/\ov{I})+\{\ov{\e}_0(I)-\ell_R(R/\ov{I}) -\ell_R(\ov{I^2}/J\ov{I})\}z+\{\ell_R(\ov{I^2}/J\ov{I})-3\}z^2+4z^3-z^4}{(1-z)^3}$,
\item[$(v)$] $\depth ~\ov{\rm G}(I)=1$, and $\H_{M'}^1(\ov{\rm G}(I))=[\H_{M'}^1(\ov{\rm G}(I))]_0$, $\ell_R([\H_{M'}^1(\ov{\rm G}(I))]_0)=1$, ${\rm a}_2(\ov{\rm G}(I))=1$, and ${\rm a}_3(\ov{\rm G}(I)) \leq -1$,
%\item[$(vii)$] $[\H_{M}^2(\ov{C})]_0 \cong R/\m$, $ [\H_{M}^2(\ov{C})]_n=(0) $ for all $n \geq 1$,and $[\H_{M}^3(\ov{C})]_n =(0)$ for all $n \geq 0$, and
\item[$(vi)$] $\ov{\rm G}(I^{\ell})$ is Cohen-Macaulay for all $\ell \geq 2$.
\end{itemize}
\end{thm}

% Let us divide our proof of Theorem \ref{d=3} in two steps.

\begin{proof}: First we prove $(1) \Rightarrow (3) \Rightarrow (2) \Rightarrow (1)$. \\
 $(1) \Rightarrow (3):$ Recall that by \cite[Theorem 1]{Itoh92} (see also \cite{HU14}) we can choose $a_1 \in I $ such that $\ov{I(R/(a_1))}=\ov{I}(R/(a_1))$, and $\ov{I^{n}(R/(a_1))}=\ov{I^n}(R/(a_1))$ for all $n \gg 0.$ Let $J$ be  a minimal reduction of $I$ such that $J=(a_1,a_2,a_3).$ Let $S=R/(a_1).$ 
We have $\ov{\e}_2(IS)=\ov{\e}_1(IS)-\ov{\e}_0(IS)+\ell_R(S/\ov{IS})+1$ by \cite[Theorem 1]{Itoh92}.
Therefore we have 
\begin{eqnarray*}
\ov{\e}_1(I)&=&\ov{\e}_1(IS) \hspace*{6.6cm}\mbox{[by \cite[Theorem 1]{Itoh92}]}\\
&=&\ov{\e}_0(IS)-\ell_{S}(S/\ov{IS})+\ell_{S}(\ov{I^2S}/J\ov{IS})+1 \hspace*{1cm} \mbox{[by Theorem \ref{maintheorem1}]}\\
&=& \ov{\e}_0(IS)-\ell_{S}(S/\ov{I}S)+\ell_{S}(\ov{I^2S}/\ov{I^2}S)+\ell_{S}(\ov{I^2}S/J\ov{I}S)+1   \\
&=& \ov{\e}_0(I)-\ell_{R}(R/\ov{I})+\ell_{R}(\ov{I^2}/J\ov{I})+\ell_{R}(\ov{I^2S}/\ov{I^2}S)+1 \hspace*{0.5cm} \ \ \ (*). 
\end{eqnarray*}
The last equality is true because $\ell_R(R/\ov{I})=\ell_{S}(S/\ov{I}S),$ and $\ell_R(\ov{I^2}/J\ov{I})=\ell_R(\ov{I^2}S/J\ov{I}S)$ $($notice that $a_1t$ is $\ov{G}(I)$-regular$)$.
We also have $\ell_R(\ov{I^3S}/J\ov{I^2 S})=1$ and $\ov{I^{n+1}S}=J\ov{I^nS}$ for all $n \geq 3$ by Theorem \ref{maintheorem1}.

\begin{claim}\label{claim1}
We have $\ell_R(\ov{I^2S}/\ov{I^2}S)=1$ and $\ov{I^nS}=\ov{I^n}S$ for all $n \geq 3$.
\end{claim}
\begin{proof}[Proof of Claim \ref{claim1}]
Assume that $\ov{I^2S}=\ov{I^2}S.$ Then we have
$\ov{\e}_1(I) = \ov{\e}_0(I)-\ell_R(R/\ov{I})+\ell_R(\ov{I^2}/J\ov{I})+1$
by the equality $(*)$.
Then, since $\ov{\e}_2(I)=\ell_R(\ov{I^2}/J\ov{I})+2,$ we have $\ov{\e}_3(I)=1$ by Theorem \ref{maintheorem1} which contradicts that $\ov{\e}_3(I)=0$.
Thus, we have $\ov{I^2S} \neq \ov{I^2}S$.

% By the same argument as in the proof of Theorem \ref{key} in the case $d=3$, we have the equality (\textcolor{red}{notice that we do not $\ov{\e}_3(I) \neq 0$ for this}) \textcolor{blue}
Since $\ov{\rm{r}}_{JS}(IS)=3$, we have $[\H_{N'}^2(\ov{G}(IS))]_n=(0)$ for all $n \geq 2$ by \cite[Proposition 3.2]{HZ94}. 
Therefore by using the long exact sequence of local cohomology modules 
\[
\cdots \longrightarrow \H^2_{N'}(\ov{\mathcal{R}'}(IS))_{n+1} \longrightarrow \H^2_{N'}(\ov{\mathcal{R}'}(IS))_{n} \longrightarrow \H_{N'}^2(\ov{G}(IS))_n \longrightarrow 0
\]
induced from
$$0 \longrightarrow \ov{\mathcal{R}'}(IS)(1) \overset{t^{-1}}{\longrightarrow} \ov{\mathcal{R}'}(IS) \longrightarrow \ov{\rm G}(IS) \longrightarrow 0$$
we obtain that $[\H_{N'}^2(\ov{\mathcal{R}'}(IS))]_n=(0)$ for all $n \geq 2.$ Therefore by Lemma \ref{rem}(2),
$\ov{\e}_3(I) = 1-\sum_{n \geq 2}\ell_R(\ov{I^nS}/\ov{I^n}S). $
Then because $\ov{\e}_3(I) =0$ by our assumption and $\ov{I^2S} \neq \ov{I^2}S$, we get  $\ell_R(\ov{I^2S}/\ov{I^2}S)=1$ and $\ov{I^nS}=\ov{I^n}S$ for all $n \geq 3$ as required.
\end{proof}

Since $\ell_R(\ov{I^2S}/\ov{I^2}S)=1$ by Claim \ref{claim1}, we choose $y \in R$ such that $\ov{I^2S}=yS+\ov{I^2}S \neq \ov{I^2}S$ and $\m S \cdot y \subseteq \ov{I^2}S$.
%We have $\ov{C}_{JR'}(IR') \cong B'(-2)$ as graded $T$-modules by the hypothesis of induction on $d$ where $B'={\mathrm R}(J')/\m {\mathrm R}(J') \cong R/\m[X_1,X_2,\ldots,X_{d-1}]$ is a polynomial ring with $d-1$ indeterminates over the field $R/\m$.
Recall that we have $\ell_R(\ov{I^3S}/J\ov{I^2S})=1$ and $\ov{I^{n+1}S}=J\ov{I^nS}$ for all $n \geq 3.$  
Since $\ov{I^3S}=\ov{I^3}S$, there exists $x \in \ov{I^3}$ such that $\ov{I^3S}=xS+J\ov{I^2S}$.
Then, because $\ov{C}_{JS}(IS) \cong B'(-2)$ by Theorem \ref{maintheorem1} where $B'={\mathcal R}(J S)/\m {\mathcal R}(J S)$, we have $\ov{C}_{JS}(IS) =\ov{xt^2}'B'$ where $\ov{xt^2}'$ denotes the image of $xt^2$ in   $\ov{C}_{JS}(IS)$.
Hence 
$$ \ov{I^3}S=\ov{I^3S}= xS+J\ov{I^2S}=(x,a_2 y,a_3 y)S+J\ov{I^2}S.$$
Because $\ov{I^3}S/J\ov{I^2}S$ is a homomorphic image of $\ov{I^3}/J\ov{I^2}$, there exist elements $y_2, y_3 \in \ov{I^3}$ which corresponds to $a_2y$, and $a_3y$ in $S$ respectively.
We then have the equality $\ov{I^3}S=(x,y_2,y_3)S+J\ov{I^2}S$ so that the equality
$$\ov{I^3}=(x,y_2,y_3)R+J\ov{I^2}+(a_1) \cap \ov{I^3}=(x,y_2,y_3)R+J\ov{I^2}$$
holds true.
We furthermore have, for all $n \geq 3$, $\ov{I^{n+1}}=J\ov{I^n}+(a_1) \cap \ov{I^{n+1}}=J\ov{I^n}$, because $\ov{I^{n+1}}S=\ov{I^{n+1}S}=J\ov{I^nS}=J\ov{I^n}S$ for all $n \geq 3$.
Therefore $\ov{C}=(\ov{xt^2},\ov{y_2t^2},\ov{y_3t^2})T$ where $\ov{xt^2}$, $\ov{y_2t^2}$, and $\ov{y_3t^2}$ denote the images of $xt^2$, $y_2t^2$, and $y_3t^2$ in $\ov{C}$, respectively.
Let $C'=(\ov{y_2t^2},\ov{y_3t^2})T$ be the graded submodule of $\ov{C}$.
Then for all $n \in \Z,$ we have
\[
 [C']_n = \begin{cases}
                   (0) & \mbox{ if } n \leq 1 \\
                   [(y_2,y_3)J^{n-2}+J^{n-1}\ov{I^2}]/J^{n-1}\ov{I^2} & \mbox{ if } n \geq 2.
                  \end{cases}
\]

\begin{claim}\label{claim2}
We have $C'_2 \cong (R/\m)^2$ so that $C'$ forms a graded $B$-module.
\end{claim}

\begin{proof}[Proof of Claim \ref{claim2}]
We have $C'_2 \cong [(y_2,y_3)+J\ov{I^2}]/J\ov{I^2}.$ 
For $u_2,u_3 \in R$ define a surjective map $ f: (R/\m)^2 \to C_2'$ as $f(\ov{u_2},\ov{u_3})=\ov{u_2y_2 + u_ 3y_3 }.$ Here $\ov{(.)}$ denotes the image of an element in respective quotient.   
It is enough to prove that the map $f$ is injective.

Notice that $f$ is well-defined. In fact, since $\m S \cdot y \subseteq \ov{I^2}S, $ we have $\m \cdot (y_2,y_3)S=\m \cdot (a_2 y,a_3 y)S \subseteq J\ov{I^2}S.$ Hence $\m \cdot (y_2,y_3) \subseteq J\ov{I^2}+(a_1) \cap \ov{I^3}=J\ov{I^2}$.

Suppose that $f(\ov{u_2},\ov{u_3})=0$ for $u_2,u_3 \in R.$ Then $u_2y_2+u_3y_3 \in J\ov{I^2}.$ 
Therefore ${u_2}'{y_2}'+{u_3}' {y_3}' = {u_2}'(a_2{y})'+{u_3}'(a_3{y})'\in J\ov{I^2}S$ where ${(.)}'$
% $\widetilde{y_3}$, and $\widetilde{y}$ 
denotes the image of an element in $S.$ 
Write $u_2'(a_2y)'+u_3'(a_3y)'=a_2'i_2'+a_3'i_3'$ with $i_2,i_3 \in \ov{I^2}$.
Then $a_2'((u_2 y)'-i_2')=a_3'(i_3'-(u_3y)')$ so that $(u_2y)'-i_2' \in (a_3') \cap \ov{I^2 S} = (a_3') \ov{I}S\subseteq \ov{I^2}S$.
Hence $(u_2y)' \in \ov{I^2}S$.
Since $y' \notin \ov{I^2}S$ we get that $u_2 \in \m$. 
By the same argument we get that $u_3 \in \m.$ Hence $f$ is injective.
Thus $f$ is an isomorphism which proves the claim.
% Thus we get a required isomorphism $(R/\m)^2 \cong [(y_2,y_3)+J\ov{I^2}]/J\ov{I^2}$.
\end{proof}

%For all $n \in \Z$, we have
%\[
 %[\ov{C}/C']_n = \begin{cases}
%                   (0) & \mbox{ if } n \leq 1 \\
%                   \ov{I^{n+1}}/[(y_2,y_3)J^{n-2}+J^{n-1}\ov{I^2}] & \mbox{ if } n \geq 2.
%                  \end{cases}
%\]
%Since $\ov{C}_{J S}(I S)$ is $B'$-module and $C'$ is a $B$-module, it follows that $\ov{C}/C'$ is a $B$-module. {\color{red}This is not clear and we don't need it, because we can directly prove that $\overline{C}/C'$ forms $B-$module by the following sentences.}
Since $\overline{C}$ is generated by the homogeneous elements of degree two, so is $\overline{C}/C'$.
We have $\m \overline{I^3} \subseteq (y_2,y_3)+J\ov{I^2}$ because $\m \overline{I^3S} \subseteq J\overline{I^2S}$ $($recall that $\ell_R(\ov{I^3S}/J\ov{I^2S})=1)$.
Then because $[\overline{C}/C']_2 \cong \ov{I^3}/[(y_2,y_3)+J\ov{I^2}]$ and $\ov{I^{3}}=(x)+[(y_2,y_3)+J\ov{I^2}]$, $\overline{C}/C'$ forms a graded $B$-module and $\ov{C}/C'=\ov{xt^2}'' B$ where $\ov{xt^2}''$ denotes the image of $xt^2$ in $\ov{C}/C'$.  

By Claim \ref{claim1} and the equality $(*)$ we have 
$\ov{\e}_1(I)= \ov{\e}_0(I)-\ell_{R}(R/\ov{I})+\ell_{R}(\ov{I^2}/J\ov{I})+2$
and hence  $\ell_{T_{\p}}(\ov{C}_{\p})=2$ by Proposition \ref{C_2}\eqref{C_2(3)}.
Consider the canonical exact sequence
$$ 0 \longrightarrow C' \longrightarrow \ov{C} \longrightarrow \ov{C}/C' \longrightarrow 0 \ \ \ (**)$$
of graded $T$-modules.
Then $\ell_{T_{\p}}(\ov{C}_{\p})=2$ implies that $1 \leq \ell_{T_{\p}}(C'_{\p}) \leq 2$.

Suppose that $\ell_{T_{\p}}(C'_{\p})=1.$ 
Then, because $\Ass_TC' \subseteq \Ass_T \ov{C}=\{\p\}$, $C'$ is a $B$-torsion free module of rank one.
Hence $C' \cong \a(m)$ as graded $B$-modules for some graded ideal $\a$ in $B$ and for some $m \in \Z$.
Since $C'$ is not $B$-free $($notice that $C'_n=(0)$ for all $n \leq 1$ and $\ell_R(C'_2)=2)$ and $B$ is a UFD, we may assume that $\height_B ~\a \geq 2$.
On the other hand we have $\ov{C}/C' \cong B(-2)$, because there is an epimorphism $B(-2) \to \ov{C}/C'$ of graded $B$-modules, $B$ is a domain, and $\dim_B\ov{C}/C'=\dim B$.
Let $P \in  \Ass_B (B/\a)$, then thanks to the exact sequence
$$ 0 \longrightarrow \a_P \longrightarrow B_P \to (B/\a)_P \longrightarrow 0 $$
we get $\depth_{B_P} C'_P=1$ because $C' \cong \a (m)$ and $\depth ~B_P \geq 2$.
Furthermore, the sequence
$$ 0 \longrightarrow C'_P \longrightarrow \ov{C}_P \longrightarrow (\ov{C}/C')_P \longrightarrow 0 $$
is exact.
Then, since $\ov{C}$ satisfies the Serre's property $(S_2)$ as $T/\Ann_T\ov{C}$-module by Corollary \ref{S2}, $\depth_{B_{P}}(\ov{C}/C')_P=0$. This gives a contradiction, because $\ov{C}/C' \cong B(-2)$ and $\height_B~ P \geq 2$.

Thus we get $\ell_{T_{\p}}(C'_{\p}) = 2$.
Then, since $C'=(\ov{y_2t^2}, \ov{y_3t^2})B$, the natural surjective map $B(-2) \oplus B(-2) \to C' \to 0$ of graded $T$-modules forms an isomorphism.
 {Therefore we have $\depth_T~ \ov{C}/ C' \geq 2$ by the exact sequence $(**)$ because the ring $B$ is $3$-dimensional Cohen-Macaulay and  $\depth_T \, \ov{C} \geq 2$} by Proposition \ref{C_2}\eqref{C_2(5)}.
Then for all $n \geq 0$, we have
\begin{eqnarray*}
\ell_R(\ov{C}_n) = 2\ell_R(B_{n-2})+\ell_R([\ov{C}/C']_n)
%&=& 2\binom{n}{2}+\ell_R([\ov{C}/C']_n)\\
= 2\binom{n+2}{2}-4\binom{n+1}{1}+2+\ell_R([\ov{C}/C']_n).
\end{eqnarray*}
Hence by Proposition \ref{C_2}\eqref{C_2(2)}, we get
$\ell_R(R/\ov{I^{n+1}}) =  \ov{\e}_0(I)\binom{n+3}{3}-\{\ov{\e}_0(I)-\ell_R(R/\ov{I})+\ell_R(\ov{I^2}/J\ov{I})+2\}\binom{n+2}{2} 
+ \{\ell_R(\ov{I^2}/J\ov{I})+4\}\binom{n+1}{1}-2-\ell_R([\ov{C}/C']_n).
$
On the other hand, since $\ov{\e}_1(I)=\ov{\e}_0(I)-\ell_R(R/\ov{I})+\ell_R(\ov{I^2}/J\ov{I})+2$, 
$\ov{\e}_2(I)=\ov{\e}_1(I)-\ov{\e}_0(I)+\ell_R(R/\ov{I})+1=\ell_R(\ov{I^2}/J\ov{I})+3,$
and $\ov{\e}_3(I) =0,$ we have $\ell_R([\ov{C}/C']_n)=\binom{n+1}{1}-2$ for all $n \gg 0$.
Hence $\ov{C}/C'$ is a $2$-dimensional Cohen-Macaulay $B$-module with multiplicity one.

In the rest of this proof, we show that $\overline{C}/C' \cong (B/a_1tB)(-2)$ as graded $T$-modules.
%\begin{proof}[Proof of Claim \ref{claim3}]

Let $\beta:R[t] \to S[t]$ be the natural $R$-algebra map defined by $\beta(t)=t.$
This induces the homomorphism 
% Then $\beta(\ov{\mathcal{R}}(I)_n)=\ov{\mathcal{R}}(IR')_n$ and $\beta(T_n)=\mathcal{R}(JR')_n$ for all $n \geq 0$, so that two surjective maps $\ov{\mathcal {R}}(I) \to \ov{\mathcal {R}}(IR')$ and $T \to {\mathcal{R}}(JR')$ induce the epimorphism 
$\psi:\overline{C} \to \ov{C}_{JS}(IS)$ of graded $T$-modules.
Since $\beta(y_2), \beta(y_3) \in J\ov{I^2S}$ and $C'=(\ov{y_2t^2}, \ov{y_3t^2})B$, $\psi$ in turn induces  the graded homomorphism $\ov{\psi}:\ov{C}/C' \to \ov{C}_{JS}(IS)$ of graded $T$-modules.
Let $\varphi:B(-2) \to \ov{C}/C'$ and $\varphi':(B/a_1tB)(-2) \to \ov{C}_{JS}(IS)$ denote homomorphisms of graded $B$-modules defined by $\varphi(1)=\overline{xt^2}''$ and $\varphi'(1)=\overline{xt^2}'$.
Consider the following commutative diagram
\[\begin{array}{ccc}
B(-2) & \overset{\varphi}{\longrightarrow} & \ov{C}/C' \\
\mapdown{i} & &\mapdown{\ov{\psi}}  \\
(B/a_1tB)(-2) & \overset{\varphi'}{\longrightarrow} & \ov{C}_{JS}(IS) 
\end{array}\]
of graded $B$-modules where $i:B \to B/a_1tB$ denotes the natural map.
Then, since $\varphi'$ is an isomorphism, we get $[(0):_B \ov{xt^2}''] \subseteq a_1t B$.

Since, $\ov{C}/C' \cong (B/[(0):_B \ov{xt^2}''])(-2)$ is a 2-dimensional Cohen-Macaulay $B$-module with multiplicity one, the natural surjective map $(B/[(0):_B \ov{xt^2}'']) \to (B/a_1tB)$ is an isomorphism. 
This completes the proof of the implication $(1) \Rightarrow (3)$ of Theorem \ref{d=3}.

$(3) \Rightarrow (2)$ and $(i)$:
Let us consider the exact sequence
$$  0 \to B(-2)^{\oplus2} \overset{\phi}{\to} \ov{C} \to (B/(a_1t))(-2) \to 0 $$
of graded $T$-modules.
Since $\overline{C}/\Im ~\phi \cong (B/a_1tB)(-2)$, we have $a_1t \overline{C} \subseteq \Im \phi \cong B(-2)^{\oplus 2}$.
Hence we have $a_1t \m \overline{C} =(0)$, so that $\m \overline{C}=(0)$ because $a_1t$ is a $\overline{C}$-regular element.
Thus, $\overline{C}$ forms a graded $B$-module.
We have
\begin{eqnarray*}
\ell_R(\overline{C}_n) = 2\ell_R(B_{n-2})+\ell_R([B/a_1tB]_{n-2})
= 2\binom{n}{2}+\binom{n-1}{1}
= 2\binom{n+2}{2}-3\binom{n+1}{1}
\end{eqnarray*}
for all $n \geq 1$.
Then since $\ell_R(\overline{C}_2)=3$ and $\overline{C}=\overline{C}_2 \cdot B$ by the above exact sequence, we have $\depth_B \overline{C}=2$ and hence the minimal $B$-free resolution
$$ 0 \to B(-m) \to B(-2)^ {\oplus 3} \to \ov{C} \to 0  $$
of $\overline{C}$ as graded $B$-module for some integer $m \geq 3$.
Then we have 
\begin{eqnarray*}
\ell_R(\overline{C}_n) 
%= 3\ell_R(B_{n-2})-\ell_R(B_{n-m}) 
= 3\binom{n}{2}-\binom{n-m+2}{2}
= 2\binom{n+2}{2}-\{6-m\}\binom{n+1}{1}+3-\binom{m}{2}
\end{eqnarray*}
for all $n \geq m-2$ so that $m=3$.

% \noindent
$(2) \Rightarrow (1)$, $(iii)$, and $(iv)$: 
We have $\ell_R(\ov{C}_n) = 2\binom{n+2}{2}-3\binom{n+1}{1}$ for $n \geq 1$ and $\depth_T\overline{C}=2$ by the exact sequence of our assertion $(2)$. Thus assertions $(1)$, $(iii)$, and $(iv)$ follow by Proposition \ref{C_2}\eqref{C_2(2)}.

$(ii)$: Since $\m \ov{C}=(0)$, $\ov{C}=\ov{C}_2 B$, $\ov{C}_2 \cong \ov{I^3}/J\ov{I^2}$, and $\ell_R(C_2)=3$, we have $\m \ov{I^3} \subseteq J\ov{I^2}$ and $\ov{I^{n+1}}=J\ov{I^n}$ for all $n \geq 3$  by Lemma \ref{fact}, and $\ell_R(\ov{I^3}/J\ov{I^2})=3$.

%$(vii)$: Let us consider the exact sequence
%$$ 0 \to [\H_{N'}^2(\ov{C})]_n \to [\H_{B_+}^3(B)]_{n-3} \to [\H_{B_+}^3(B)]_{n-2}^{\oplus 3} \to [\H_{N'}^3(\ov{C})]_n \to 0$$
%of local cohomology modules for $n \in \Z$ which is induced by the exact sequence of our assertion $(2)$. 
%Then, because $[\H_{B_+}^3(B)]_{-3} \cong R/\m$ and $[\H_{B_+}^3(B)]_{n}=(0)$ for $n \geq -2$, we have $[\H_{N'}^2(\ov{C})]_n=(0)$ for all $n > 0$, $[\H_{N'}^2(\ov{C})]_0 \cong R/\m$, and $[\H_{N'}^3(\ov{C})]_n=(0)$ for all $n \geq 0$.

$(v)$ and $(vi)$:
Since $\ov{\e}_3(I)=0,$ by \cite[Theorem 4.1, Lemma 4.7]{Bla97} we have $\H^3_{N'}(\ov{\mathcal{R}'}(I))_n=0$ for all $n \geq 0.$ 
We have $\ell_R(R/\overline{I^{n+1}})=\sum_{i=0}^3(-1)^i\overline{\e}_i(I)\binom{n+3-i}{3-i}$ for all $n \geq 1$ by Proposition \ref{C_2}(2) because $\ell_R(\ov{C}_n) = 2\binom{n+2}{2}-3\binom{n+1}{1}$ for all $n \geq 1$ as above.
Hence, by \cite[Theorem 4.1]{Bla97}, we have $[\H_{N'}^2(\ov{\mathcal{R}'}(I))]_n=(0)$ for all $n \geq 2$ so that $\H_{N'}^2(\ov{\mathcal{R}'}(I))=[\H_{N'}^2(\ov{\mathcal{R}'}(I))]_1$ by Theorem \ref{Itoh2}(3).
We also get $\ell_R([\H_{N'}^2(\ov{\mathcal{R}'}(I))]_1)=\sum_{i=0}^3(-1)^i\overline{\e}_i(I)-\ell_R(R/\overline{I})=1$ by \cite[Theorem 4.1]{Bla97} and our assumption. 

Now consider the short exact sequence
$$ 0 \longrightarrow \ov{\mathcal{R}'}(I)(1) \longrightarrow \ov{\mathcal{R}'}(I) \to \ov{G}(I) \to 0.$$
By using the induced long exact sequence of local cohomology modules 
\begin{eqnarray*}
 0 \to  \H^1_{N'}(\ov{G}(I))_n \to \H^2_{N'}(\ov{\mathcal{R}'}(I))_{n+1} \to \H^2_{N'}(\ov{\mathcal{R}'}(I))_{n} \to  \H^2_{N'}(\ov{G}(I))_n \\\to \H^3_{N'}(\ov{\mathcal{R}'}(I))_{n+1} \to \H^3_{N'}(\ov{\mathcal{R}'}(I))_{n} \to 
 \H^3_{N'}(\ov{G}(I))_n \to 0, 
\end{eqnarray*}
we get $(v)$. The assertion $(vi)$ follows by Corollary \ref{cor3}.
% \end{proof}
\end{proof}

\begin{rem} 
% (1) It would be interesting to find an example verifying the assumptions of Theorem \ref{d=3}.   In the Gorenstein case this would produce   a counterexample to Itoh's conjecture. Otherwise we hope that one of the equivalent conditions of Theorem \ref{d=3} will be useful to prove Itoh's conjecture at least in this case.
\noindent  We notice that Theorem \ref{d=3} can be extended to $d \ge 3$ under the assumption $\depth ~\ov G(I) \ge d-2.$ We omit the details because the techniques are standard. 
\end{rem}


\begin{thebibliography}{99}

\bibitem{Bla97} C. Blancafort, {\em On Hilbert functions and cohomology}, J. Algebra {\bf 192} (1997), 439--459.

\bibitem{CPR05} A.~Corso, C.~Polini and M.~E. Rossi, {\em Depth of the  associated graded rings via   {H}ilbert coefficients of ideals}, Journal of Pure and Applied Algebra {\bf 201} (2005),  126-141.

\bibitem{CPR16} A.~Corso, C.~Polini and M.~E. Rossi, {\em Bounds on the normal {H}ilbert coefficients}, Proc. Amer. Math. Soc. {\bf 144} (2016), 1919-1930.  
 
%\bibitem[EV91]{EV91} J. Elias and G. Valla, {\em Rigid Hilbert functions}, J. Pure and Appl. Algebra {\bf 71} (1991), 19--41.
 
%\bibitem[GNO08]{GNO08} S. Goto, K. Nishida, and K. Ozeki, \textit{Sally modules of rank one}, Michigan Math. J. $\bf57$ (2008) 359--381.

% \bibitem[GR98]{GR98} A. Guerrieri and M. E. Rossi, {\em Hilbert coefficients of Hilbert filtrations}, J. Algebra {\bf 199} (1998), 40-61.

\bibitem{GMV19} K. Goel, V. Mukundan and J. K. Verma, {\em On the Vanishing of the normal Hilbert coefficients of ideals}, Preprint, available at arxiv1901.06310.

\bibitem{HZ94} L. T. Hoa and S. Zarzuela, \textit{Reduction number and $\rma$-invariant of good filtrations}, Comm. Algebra {\bf 22} (1994), 5635--5656.

\bibitem{HU14} J.~Hong and B.~Ulrich, {\em Specialization and integral closure}, J. Lond. Math. Soc. (2) {\bf 90(3)} (2014), 861--878.


\bibitem{HH99} S. Huckaba and C. Huneke, {\em Normal ideals in regular rings}, J. Reine Angew. Math. {\bf 510} (1999), 63--82.

 
 \bibitem{HM97} S. Huckaba and T. Marley, {\em Hilbert coefficients and the depths of associated graded rings}, J. London Math. Soc. {\bf 56} (1997), 64-76.

 

\bibitem{H87} C. Huneke, \textit{Hilbert functions and symbolic powers}, Michigan Math. J. ${\bf 34}$ (1987), 293--318.


 \bibitem{Itoh88} S.~Itoh, {\em Integral closures of ideals generated by regular sequences}, 
J. Algebra, {\bf 117(2)} (1988), 390--401.
 
 \bibitem{Itoh92} S.~Itoh, {\em Coefficients of normal {H}ilbert polynomials}, J. Algebra {\bf 150(1)} (1992), 101--117.
 
\bibitem{KM15} M. Kummini and S. K. Masuti, {\em On conjectures of Itoh and of Lipman on the cohomology of normalized blow-ups}, communicated, available at arxiv1507.03343.

\bibitem{LT81} J. Lipman and B. Teissier, {\em Pseudo-rational local rings and a theorem of Brian\c on-Skoda about integral closure of ideals}, Michgan Math. J. {\bf 28} (1981), 97--116.

% \bibitem{Mar89} T. Marley, {\em Hilbert functions of ideals in Cohen-Macaulay rings}, Ph. D. Thesis, Purdue University, (1989).  
 
\bibitem{MOR17} S. K. Masuti, K. Ozeki, and M. E. Rossi, {\em On the  normal Sally module and the first Hilbert coefficient}, to appear in J. Algebra (2018).

\bibitem{N60} D. G. Northcott, \textit{A note on the coefficients of the abstract Hilbert function}, J. London Math. Soc. ${\bf 35}$ (1960), 209--214.

\bibitem{OWY17} T. Okuma, K.-i. Watanabe, and K.~Yoshida, {\em Rees algebras and $p_g$-ideals in a two-dimensional normal local domain}, Proc. Amer. Math. Soc. ${\bf 145}$  (2017), 39--47.

\bibitem{OWY19} T. Okuma, K.-i. Watanabe, and K.~Yoshida, {\em Normal reduction numbers for normal surface singularities with application to elliptic singularities of Brieskorn type}, To appear in Acta Mathematica Vietnamica (2019).

\bibitem{O87} A. Ooishi, \textit{$\Delta$-genera and sectional genera of commutative rings}, Hiroshima Math. J. $\bf17$ (1987), 361--372.

%\bibitem{O87_} A. Ooishi, \textit{Genera and arithmetic genera of commutative rings}, Hiroshima Math. J. $\bf17$, 1987, 47--66.

\bibitem{OR16} K. Ozeki and M. E. Rossi, {\em The structure of the Sally module of integrally closed ideals}, Nagoya Math. J. {\bf 227} (2017), 49--76.

\bibitem{Phu18} T. T. Phuong, {\em Normal Sally modules of rank one}, J. Algebra {\bf 493} (2018), 236--250.

\bibitem{Ree61} D. Rees, \textit{A note on analytically unramified local rings}, J. London Math. Soc. ${\bf 36}$ (1961), 24--28.

\bibitem{Ree81} D. Rees, \textit{Hilbert functions and pseudo-rational local rings of dimension two}, J. London Math. Soc. ${\bf 24}$ (1981), 467--479.

\bibitem{RV10} M. E. Rossi and G. Valla \textit{Hilbert functions of filtered modules}, UMI Lecture Notes 9, Springer (2010).


\bibitem{Tru87} N. V. Trung, \textit{Reduction exponent and degree bound for the defining equations of graded rings}, Proc. Amer. Math. Soc. {\bf{101}} (1987), 229--236.

\bibitem{V94} W. V. Vasconcelos, \textit{Hilbert functions, analytic spread, and Koszul homology}, Contemporary Mathematics, {\bf 159} (1992), 401--422.

\bibitem{VP96} M. Vaz Pinto, \textit{Hilbert functions and Sally modules}, J. Algebra, {\bf 192} (1997), 504--523.
\end{thebibliography}
\end{document}